\documentclass[a4paper,12pt]{article}
\usepackage{cmap}
\usepackage[cp1251]{inputenc}
\usepackage[english]{babel}
\usepackage[left=2cm,right=2cm,top=2cm,bottom=2cm]{geometry}
\usepackage{amssymb}
\usepackage{amsmath, amsthm}
\theoremstyle{plain}
\newtheorem{theorem}{Theorem}
\newtheorem{lemma}{Lemma}
\newtheorem{proposition}{Proposition}
\newtheorem{corollary}{Corollary}

\theoremstyle{definition}
\newtheorem{example}{Example}
\newtheorem{remark}{Remark}
\newtheorem{definition}{Definition}

\numberwithin{equation}{section}

\begin{document}

\begin{center}\Large
\textbf{Finite groups and $k$-submodular subgroups}
\normalsize

\bigskip
T. I. Vasilyeva

\end{center}
\bigskip

\begin{center}
{\bf Abstract}
\end{center}


{\small~~~

For natural numbers $n$ and $k$, the concepts of $n$-modularly embedded subgroup, $k$-submodular subgroup and $k$-$\mathrm{LM}$-group are given, which generalize, respectively, the concepts of modular subgroup, submodular subgroup and $\mathrm{LM}$-group. Classes of groups with given systems of $k$-submodular subgroups are investigated.

\medskip

}

\medskip

\smallskip

MSC2010\  20D10, 20E15, 20F16



\bigskip


\section{Introduction}

All the groups under consideration are finite. The modular subgroup \cite{sch2}, as a modular element in the lattice of all subgroups of a group, is one of the generalizations of a normal subgroup.
A subgroup $M$ of a group $G$ is called {\sl modular} in $G$ if the following are provided that
(1)~$\langle X, M \cap Z \rangle = \langle X, M \rangle \cap Z$ for all $X \leq G$, $Z \leq G$ such that $X \leq Z$ and
(2)~$\langle M, Y \cap Z\rangle  = \langle M, Y \rangle  \cap Z$ for all $Y \leq G$, $Z \leq G$ with $M \leq Z$.

Modular, like normal subgroups, do not have the property of transitivity.  However, the concept of a submodular subgroup introduced in \cite{Zim} is free of this drawback. A subgroup $H$ of a group $G$ is called {\sl submodular} in $G$ if $H$ can be connected to $G$ by a chain of subgroups in which each previous subgroup is modular in the next. In \cite{Zim}, the properties of such subgroups were studied. In \cite{VVA2015}, classes of groups with submodular Sylow subgroups were investigated.
Currently, groups with given systems of submodular subgroups are studied
in the works of many authors (see, for example, [4--10]). 

R.~Schmidt in \cite{sch1}
proved that the subgroup $M$ of the group $G$ is a maximal modular subgroup in $G$ if and only if either $M$ is a maximal normal subgroup of $G$ or $G/\mathrm{Core}_{G}(M)$ is a non-abelian group of order $pq$ for some primes $p$ and $q$. Here $\mathrm{Core}_{G}(M)$ is the intersection of all subgroups of $G$ conjugate to $M$.
We introduce the following

\begin{definition}\label{d1}
Let $n$ be a natural number.
A subgroup $H$ of a group $G$ will be called $n$-modularly embedded in $G$ if either $H\unlhd G$ or $G/\mathrm{Core}_{G}(M)$ is a non-nilpotent group of order $pq^{n}$ with $|G:H| =p$ for some primes $p$ and $q$.
\end{definition}

\begin{definition}\label{d2}
Let $k$ be a fixed natural number. A subgroup $H$ of a group $G$ will be called $k$-submodular in $G$ if there exists a chain of subgroups
\begin{equation} \label{1_1}
		H = H_0 \leq H_1\leq \cdots \leq H_{m-1}\leq H_m=G,
\end{equation}
such that $H_{i-1}$ is $n_{i}$-modularly embedded in $H_{i}$ for some natural number $n_{i}\leq k$ and every $i = 1, \ldots , m$.
\end{definition}

\begin{example} \label{e3.1}
Let $n=2$ and let $G=\langle a, b$ $|$ $a^{5}=b^{4}=1, ab=ba^{3}\rangle$ be the holomorph of the cyclic group $X= \langle a\rangle$. The subgroup $Y=\langle b\rangle$ has $\mathrm{Core}_{G}(Y)=1$. Since $|G|=5\cdot2^2$ and $|G:Y|=5$ it follows that $Y$  is $2$-modularly embedded in $G$.

Let $k=2$. The subgroup $H=\langle b^2 \rangle$ is $2$-submodular in $G$ since there exists the chain of subgroups $H \unlhd Y < G$.
\end{example}

\begin{remark}\label{r1}
Note that the following statements follow from the above-mentioned result of R.~Schmidt and Definitions \ref{d1}, \ref{d2} in the case $k=1$.

(1) For a maximal subgroup $M$ of a group $G$, $M$ is $1$-submodular in $G$ if and only if $M$ is modular in $G$.

(2) A subgroup $H$ of $G$ is $1$-submodular in $G$ if and only if $H$ is submodular in $G$.
\end{remark}

In this article, properties of $k$-submodular subgroups of groups and classes of groups with given systems of $k$-submodular subgroups are obtained.

For a natural number  $n$, let us recall that a subgroup $M$ is called $n$-maximal subgroup of $G$ if $M$ is maximal in $(n-1)$-maximal subgroup $W$ of $G$ ($G$ is $0$-maximal subgroup of $G$).

\medskip

In what follows, $k$ is a fixed natural number.

\begin{theorem} \label{t3.1}
Let $G$ be a group. Then the following statements are equivalent.

$(1)$ Every maximal subgroup of $G$ is $k$-submodular in $G$.

$(2)$ $G$ is supersoluble and for all complemented chief factors $H/K$ of $G$, $G/C_{G}(H/K)$ is either $1$ or cyclic of order $q^{n}$ where $q$ is some prime and $n$ is some natural number such that $n\leq k$.

$(3)$ A group $G$ is soluble and for all maximal subgroups $M_{1}$ and $M_{2}$ of $G$ either $M_{1}= M_{2}$ or $M_{1}\cap M_{2}$ is a $n$-maximal subgroup of $M_{i}$ and $|M_{i}: M_{1}\cap M_{2}|=q_{i}^{n}$, $i=1, 2$, where $q_{1}$ and $q_{2}$ are some primes (not necessarily different) and $n$ is some natural number such that $n\leq k$.
\end{theorem}

\begin{corollary} \cite[Theorem~5.3.10]{sch2}\label{c1}
Every maximal subgroup of a group $G$ is modular in $G$ if and only if $G$ is supersoluble and for all complemented chief factors $H/K$ of $G$, $G/C_{G}(H/K)$ is either cyclic of prime order or $1$.
\end{corollary}

Recall \cite[p.~130]{Wei} that a group $G$ is called an {\sl $\mathrm{LM}$-group} if for every pair of subgroups $A$, $B$ of $G$ with $A$ maximal in $\langle A, B\rangle$, $A\cap B$ is a maximal subgroup of $B$.

\begin{definition}\label{d3}
Let $k$ be a fixed natural number. A group $G$ will be called an $k$-$\mathrm{LM}$-group if for every pair of subgroups $A$, $B$ of $G$ with $A$ maximal in $\langle A, B\rangle$, $A\cap B$ is an $n$-maximal subgroup of $B$ and $|B: A\cap B|=q^{n}$ for some prime $q$ and some natural number $n\leq k$.
\end{definition}

\begin{example} \label{e3.2}
Let $k=2$ and $n\in\{1, 2\}$. Let $G$ is from the example \ref{e3.1}. Let $A$ and $B$ denote subgroups of $G$ with $A$ maximal in $\langle A, B\rangle$. The following cases are possible.

1. $A=1, B\cong Z_{2}$; $A\cong Z_{5}, B\cong Z_{2}$;
$A\cong Z_{4}, B\cong Z_{2}$ and $B\not\leq A$; $A\cong Z_{2}, B\cong Z_{4}$ and $A\leq B$. Then $A\cap B$ is an $1$-maximal subgroup of $B$ and $|B: A\cap B|=2$.

2. $A=1, B\cong Z_{5}$; $A\cong Z_{2}, B\cong Z_{5}$;
$A\cong Z_{4}, B\cong Z_{5}$; $A\cong Z_{4}, |B|=10$;
$A\cong Z_{4}, B=G$. Then $A\cap B$ is an $1$-maximal subgroup of $B$ and $|B: A\cap B|=5$.

3. $A\cong Z_{4}, B=A^{h}$ for $1\not=h\in G$. If $|A\cap B|=2$ then $G=\langle A, B\rangle \leq N_{G}(A\cap B)$. This contradicts the fact that $A\cap B\leq\mathrm{Core}_{G}(A)=1$. Thus, $A\cap B=1$ is an $2$-maximal subgroup of $B$ and $|B: A\cap B|=2^{2}$.

So $G$ is a $2$-$\mathrm{LM}$-group, but $G$ is not an $\mathrm{LM}$-group.
\end{example}

\begin{remark}\label{r2}
 By \cite[Theorem~4.4.3]{Wei} every $\mathrm{LM}$-group $G$ is supersoluble and therefore in $G$, every maximal subgroup of any subgroup has a  prime index by \cite[Theorem~VI.9.5]{Hup}. Then the following statement holds for $k=1$.

(1) A group $G$ is $1$-$\mathrm{LM}$-group if and only if $G$ is an $\mathrm{LM}$-group.
\end{remark}

\begin{theorem} \label{t3.2}
Let $G$ be a group. Then the following statements are equivalent.

$(1)$ Every subgroup of $G$ is $k$-submodular in $G$.

$(2)$ Every maximal subgroup of any subgroup $A$ of $G$ is $k$-submodular in $A$.

$(3)$ $G$ is supersoluble and $G/C_{G}(H/K)$ is either $1$ or a cyclic group of a order $q^{n}$ for every chief factor $H/K$ of $G$, where $q$ is some prime and a natural number $n\leq k$.

$(4)$ $G$ is an $k$-$\mathrm{LM}$-group.
\end{theorem}

In \cite{Wei} and \cite{sch2} the following characterization of $\mathrm{LM}$-groups was given, based on the work of Ito \cite{Ito}.

\begin{corollary} \cite[Theorem~4.4.4]{Wei}, \cite[Theorem~5.3.11]{sch2}\label{c2}
A group $G$ is an $\mathrm{LM}$-group if and only if $G$ is supersoluble and for all chief factor $H/K$ of $G$, $G/C_{G}(H/K)$ is either cyclic of prime order or $1$.
\end{corollary}

\begin{theorem} \label{t3.3}
Let $\mathfrak{X}$  be a class of all groups $G$  such that every maximal subgroup of $G$ is $k$-submodular in $G$. Let $\mathfrak{Y}$  be a class  of all groups $G$ such that
every subgroup of $G$ is $k$-submodular in $G$.
Then

$(1)$ $\mathfrak{X}$  is a Schunk class consisting of supersoluble groups;

$(2)$ $\mathfrak{Y}$ is a hereditary formation;

$(3)$ a group $G\in\mathfrak{X}$ if and only if $G/\Phi(G)\in\mathfrak{Y}$.
\end{theorem}

\begin{corollary} \cite[Proposition~8]{Zim}\label{c3}
Every maximal subgroup of a group $G$ is modular in $G$ if and only if every subgroup of $G/\Phi(G)$ is submodular in $G/\Phi(G)$.
\end{corollary}

Further $\mathfrak{A}(p-1)_{k}$ means the class of all abelian groups of exponent dividing $p-1$ and not dividing by the $(k + 1)$th powers of primes.

\begin{theorem} \label{t3.5}
Let $\mathfrak{K}$ be the class of all supersoluble groups in which every Sylow subgroup is $k$-submodular. Then $\mathfrak{K}$ is a hereditary saturated formation and $\mathfrak{K}$  is defined by a formation function $h$ such that
$h(p)=(G\ | \ G\in\mathfrak{A}(p-1)_{k}$ and all Sylow subgroups of $G$ are cyclic$)$.
\end{theorem}

\begin{theorem} \label{t3.6}
The class $\mathfrak{F}$ of all groups in which every Sylow subgroup is $k$-submodular is  a hereditary saturated formation and $\mathfrak{F}$ is defined by a formation function $f$ such that
$f(p)=(G\ | \ \mathrm{Syl}(G)\subseteq\mathfrak{A}(p-1)_{k}$ and all Sylow subgroups of $G$ are cyclic$)$.
\end{theorem}

For $k=1$, the classes of groups $\mathfrak{K}$ and $\mathfrak{F}$ coincide with the classes $s\mathfrak{U}$ and $sm\mathfrak{U}$, respectively, studied in \cite{VVA2015}. Theorems~A and C from \cite{VVA2015} follow from Theorems~\ref{t3.5} and \ref{t3.6}.

\begin{remark}\label{r3}
$\mathfrak{Y}\subset\mathfrak{X}\subset\mathfrak{K}\subset\mathfrak{F}$, where

$\mathfrak{Y}=(G$ $|$ all subgroups of $G$ are $k$-submodular in $G)$,

$\mathfrak{X}=(G$ $|$ all maximal subgroups of $G$ are $k$-submodular in $G)$,

$\mathfrak{K}=(G$ $|$ $G$ is supersoluble and all Sylow subgroups of $G$ are $k$-submodular in $G)$,

$\mathfrak{F}=(G$ $|$ all Sylow subgroups of $G$ are $k$-submodular in $G)$.
\end{remark}

Groups with $\mathfrak{U}_{k}$-subnormal subgroups were studied in \cite{MonSok_Sib_el23}. Here $\mathfrak{U}_{k}$ is the class of all supersoluble groups in which exponents are not divided by the (k+1)th powers of primes.
If $G$ is a soluble group, then an $k$-submodular in $G$ subgroup is  $\mathfrak{U}_{k}$-subnormal in $G$ (Lemma \ref{l2.7}). The converse statement does not hold in the general case.

\begin{example} \label{e3.3}
Let $k=n=1$ and let $G=\langle x, y$ $|$ $x^{7}=y^{6}=1, xy=yx^{5}\rangle$ be the holomorph of the cyclic group $X= \langle x\rangle$. Since $\mathrm{Core}_{G}(Y)=1$ and $G\in\mathfrak{U}_{1}$ it follows that $Y=\langle y\rangle$  is $\mathfrak{U}_{1}$-subnormal in $G$. But $Y$ is not 1-submodular in $G$ because $|G|=|G/\mathrm{Core}_{G}(Y)|=2\cdot3\cdot7$.
\end{example}

At the same time, the following statement takes place.

\begin{proposition} \label{p3.1}
$(1)$ $\mathfrak{K}= \mathfrak{U}\cap\mathrm{w}\mathfrak{U}_{k}$, where
$\mathfrak{U}\cap\mathrm{w}\mathfrak{U}_{k}$ is the class of all supersoluble groups $G$ in which any Sylow subgroup of $G$ is $\mathfrak{U}_{k}$-subnormal.

$(2)$ $\mathfrak{F}= \mathrm{w}\mathfrak{K}$, where
$\mathrm{w}\mathfrak{K}$ is the class of all groups $G$ in which any
Sylow subgroup of $G$ is $\mathfrak{K}$-subnormal.
\end{proposition}

An application of the results obtained above to the products of groups was found.

\begin{theorem} \label{t3.6_1}
Suppose that a group $G = AB$ is the product of nilpotent groups $A$ and $B$. If $A$ and $B$ are $k$-submodular in $G$ then $G$ is supersoluble and every Sylow subgroup is $k$-submodular in $G$.
\end{theorem}

\section{Preliminaries}

We use the notation and terminology from \cite{DH, BalEzq}.
We recall some concepts significant in the paper.

Let $G$ be a group. If $H$ is a subgroup of $G$, we write $H\leq G$ and if $H\not= G$,
we write $H < G$.
We denote by $|G|$ the order of $G$;
by $\pi(G)$ the set of all distinct prime divisors of the order of $G$;
by ${\rm {Syl}}_p(G)$ the set of all Sylow $p$-subgroups of $G$;
by ${\rm {Syl}}(G)$ the set of all Sylow subgroups of $G$;
by $F(G)$ the Fitting subgroup of $G$;
by $|G:H|$ the index of $H$ in $G$;
by $\pi(G:H)$ the set of all different prime divisors of $|G:H|$;
by $Z_n$ the cyclic group of order $n$;
by $\mathbb{P}$ the set of all primes;
by $\mathbb{N}$ the set of all natural numbers;
by $\mathfrak{S}$ the class of all soluble groups;
by $\mathfrak{U}$ the class of all supersoluble groups;
by $\mathfrak{N}$ the class of all nilpotent groups;
by $\mathfrak{A}(p-1)$ the class of all abelian groups of exponent dividing $p-1$;
by $\mathfrak{U}_{k}$ the class of all supersoluble groups in which exponents are not divided by the $(k + 1)$th powers of primes where $k\in\mathbb{N}$;
by $\mathcal{A}_{k}$ the class of all groups with abelian  Sylow subgroups in which exponents are not divided by the $(k + 1)$th powers of primes where $k\in\mathbb{N}$.

A group $G$ of order $p_1^{\alpha_1}p_2^{\alpha_2}\cdots p_n^{\alpha_n}$
(where $p_1>p_2>\cdots >p_n$ and $p_i$ is a prime, $i=1, 2,\ldots, n$) is called {\sl Ore dispersive} 
or {\sl Sylow tower group} whenever $G$ has a normal subgroup of order $p_1^{\alpha_1}p_2^{\alpha_2}\cdots p_i^{\alpha_i}$
for every $i=1, 2,\ldots, n$.

A class of groups $\mathfrak{H}$ is called:
{\sl a homomorph} if from $G\in\mathfrak{H}$ and $N\unlhd G$ it follows that $G/N\in\mathfrak{H}$;
{\sl hereditary} if together with $G\in\mathfrak{H}$ all subgroups of $G$ belong to $\mathfrak{H}$;
{\sl saturated} if from $G/\Phi(G)\in\mathfrak{H}$ it follows that $G\in\mathfrak{H}$;
{\sl primitive closed} if from $G/\mathrm{Core}_{G}(M)\in\mathfrak{H}$ for any maximal subgroup $M$ of $G$ it follows that $G\in\mathfrak{H}$;
{\sl a formation} if $\mathfrak{H}$ is a homomorph and, for $N_{1}\unlhd G$ and $N_{2}\unlhd G$,
a group $G/(N_{1}\cap N_{2})\in\mathfrak{H}$ whenever $G/N_{i}\in \mathfrak{F}$,
$i=1, 2$; 
{\it a Schunck class} if $\mathfrak{H}$ is a primitive closed homomorph.

A subgroup $H$ of a group $G$ is said to be:

{\sl $\mathbb{P}$-subnormal} in $G$ \cite{VasVasTyu2010} if either $H= G$ or there exists a chain of subgroups (\ref{1_1})
such that $|H_{i} : H_{i-1}|$ is a prime for every $i = 1,\ldots , m$;
{\sl $\mathrm{K}$-$\mathbb{P}$-subnormal} in $G$ \cite{VasVasTyu2014} if there exists a chain of subgroups (\ref{1_1})
such that either $H_{i-1}\unlhd H_{i}$ or $|H_{i} : H_{i-1}|$ is a prime for every $i = 1,\ldots , m$.

Let $\mathfrak{F}$ be a non-empty formation. A subgroup $H$ of a group $G$ is said to be:
{\sl $\mathfrak{F}$-subnormal} in $G$ \cite{BalEzq} if either $H=G$ or there exists a chain of subgroups (\ref{1_1})
such that $H_{i}^{\mathfrak{F}} \leq H_{i-1}$ for every $i = 1,\ldots , m$;
{\sl $\mathrm{K}$-$\mathfrak{F}$-subnormal} in $G$ \cite{BalEzq} if there exists a chain of subgroups (\ref{1_1})
such that either $H_{i-1}\unlhd H_{i}$ or $H_{i}^{\mathfrak{F}} \leq H_{i-1}$ for every $i = 1,\ldots , m$.
Here $G^{\mathfrak{F}}$ is the least normal subgroup of $G$ for which $G/G^{\mathfrak{F}}\in \mathfrak{F}$.
If a subgroup $H$ of a group $G$ is $\mathfrak{F}$-subnormal in $G$, we use the designation $H$ $\mathfrak{F}$-$sn$ $G$.
A class $\mathrm{w}\mathfrak{F}=(G$ $|$ $\pi(G)\subseteq\pi(\mathfrak{F})$ and $H$ $\mathfrak{F}$-$sn$ $G$ for all $H\in\mathrm{Syl}(G))$.

A function $f:\Bbb{P}\rightarrow \{$formations$\}$ is called a {\sl formation function}.
A formation $\mathfrak{F}$ is called {\sl local},
if there exists a formation function $f$ such that $\mathfrak{F}=LF(f)=(G | G/C_{G}(H/K)\in f(p)$ for every chief factor $H/K$ of a group $G$ and for all primes $p\in\pi(H/K))$.

\medskip

\begin{lemma}\cite[Theorem A.15.6]{DH} \label{l1.1}
Let  $G$ be a soluble group $G$ with a maximal subgroup $M$ such that $\mathrm{Core}_{G}(M)=1$.

$(1)$ $G$ has a unique minimal normal subgroup $N$, $G=NM$, $N\cap M=1$ and $N=C_{G}(N)=F(G)$.

$(2)$ If $p\in\pi(N)$ then $O_{p}(M)=1$.

$(3)$ All complements to $N$ in $G$ are conjugate to $M$.
\end{lemma}

\begin{lemma} \cite[Lemma 2]{Hei} \label{l1.7}
Suppose that $G = AB$ is the product of nilpotent subgroups $A$ and $B$,
while $G$ has a minimal normal subgroup $N$ such that $N = C_{G}(N) \not= G$. Then

$(1)$ $A\cap B = 1$;

$(2)$ $N\leq A \cup B$;

$(3)$ if $N\leq A$ then $A$ is a $p$-group for some prime $p$ and $B$ is a $p'$-group.
\end{lemma}

\begin{lemma}\cite[Theorem VI.9.5]{Hup} \label{l1.2}
A group $G$ is supersoluble if and only if all maximal subgroups of $G$ have a prime index.
\end{lemma}

\begin{lemma}\cite[Example IV.3.4(f), Theorem VII.2.1]{DH} \label{l1.3}
$(1)$ If $G\in \mathfrak{U}$, then $G$ is Ore dispersive and has nilpotent commutator subgroup.

$(2)$ The class $\mathfrak{U}$ is a hereditary saturated formation and has a formation function $f$ such that $f(p) = \mathfrak{A}(p - 1)$ for every prime $p$.
\end{lemma}

\begin{lemma}\cite[Theorem 2]{Vas2004} \label{l1.4}
A group $G$ is supersoluble if and only if $G$ can be represented as the product of two nilpotent $\mathbb{P}$-subnormal subgroups.
\end{lemma}

\begin{lemma} \cite[Lemma~6.1.6]{BalEzq} \label{l1.8}
Assume that $\mathfrak{F}$ is a non-empty formation and $G$ is a group.

$(1)$ If $H$ $\mathfrak{F}$-$sn$ $K$ and $K$ $\mathfrak{F}$-$sn$ $G$ then $H$ $\mathfrak{F}$-$sn$ $G$.

$(2)$ If $N \unlhd G$ and $H/N$ $\mathfrak{F}$-$sn$ $G/N$ then $H$ $\mathfrak{F}$-$sn$ $G$.

$(3)$ If $H$ $\mathfrak{F}$-$sn$ $G$ then $HN/N$ $\mathfrak{F}$-$sn$ $G/N$.
\end{lemma}

\begin{lemma} \cite[Lemma~6.1.7]{BalEzq} \label{l1.9}
Assume that $\mathfrak{F}$ is a hereditary formation.

$(4)$ If $G^{\mathfrak{F}} \leq H$ then $H$ $\mathfrak{F}$-$sn$ $G$.

$(5)$ If $H$ $\mathfrak{F}$-$sn$ $G$ then $H\cap K$ $\mathfrak{F}$-$sn$ $K$.

$(6)$ If $H$ $\mathfrak{F}$-$sn$ $G$ and $K$ $\mathfrak{F}$-$sn$ $G$ then $H\cap K$ $\mathfrak{F}$-$sn$ $G$.
\end{lemma}

We will need a result that follows from \cite[Theorem~3.4]{VasVasVeg2016}.

\begin{lemma} \cite{VasVasVeg2016}\label{l1.10}
If $\mathfrak{F}$ is a hereditary saturated formation then $\mathrm{w}\mathfrak{F}$ is a hereditary saturated formation.
\end{lemma}

\begin{lemma} \cite{VasVasTyu2010}\label{l1.5}
$(1)$ $\mathrm{w}\mathfrak{U}=(G$ $|$ $H$ $\mathfrak{U}$-$sn$ $G$ for all $H\in\mathrm{Syl}(G))=(G$ $|$ $H$ is $\mathbb{P}$-subnormal in $G$ for all $H\in\mathrm{Syl}(G))$.

$(2)$ If $G\in \mathrm{w}\mathfrak{U}$ then $G$ is Ore dispersive.

$(3)$ The class $\rm{w}\frak {U}$ is a hereditary saturated formation and it has a formation function $f$ such that
$f(p)=(G\in\frak{S}$ $|$ ${\rm {Syl}}(G)\subseteq
\frak {A}(p-1))$ for every prime $p$.
\end{lemma}

\begin{lemma} \label{l1.6}
$(1)$ The class $\mathfrak{U}_{k}$ is a hereditary formation. \cite[Lemma 3(2)]{MonSok_Sib_el23}.

$(2)$ The class $\mathrm{w}\mathfrak{U}_{k}$ is a hereditary saturated formation. \cite[Proposition 1]{MonSok_Sib_el23}.

$(3)$ If $G\in \mathrm{w}\mathfrak{U}_{k}$, then $G/F(G)\in \mathcal{A}_{k}$ \cite[Corollary 1]{MonSok_Sib_el23}.
\end{lemma}

\section{Properties of $k$-submodular subgroups of groups}

\begin{lemma} \label{l2.1}
Let $H$ be a subgroup of a group $G$ and $N\leq H$ for some $N\unlhd G$. The subgroup $H$ is $n$-modularly embedded in $G$ if and only if $H/N$ is $n$-modularly embedded in $G/N$.
\end{lemma}

\begin{proof}

Note that $\mathrm{Core}_{G}(H)/N=\mathrm{Core}_{G/N}(H/N)$. Then the proof is obtained by direct verification of Definition \ref{d1}.
\end{proof}

\begin{lemma} \label{l2.2}
Let $M$ be a maximal subgroup of a group $G$. Then the following statements are equivalent.

$(1)$ $M$ is $k$-submodular in $G$.

$(2)$ Either $M\unlhd G$, or $G/\mathrm{Core}_{G}(M)$ is a non-nilpotent biprimary group with a cyclic Sylow $q$-subgroup $M/\mathrm{Core}_{G}(M)$ and a normal Sylow $p$-subgroup $P/\mathrm{Core}_{G}(M)$ such that $|M/\mathrm{Core}_{G}(M)|=q^{n}$ and $|P/\mathrm{Core}_{G}(M)|=p$ for some $p, q\in \mathbb{P}$ and a natural number $n\leq k$.
\end{lemma}

\begin{proof}
Assume that $(1)$ holds. Let $D=\mathrm{Core}_{G}(M)$.
If $M= D$, then $M\unlhd$ $G$. Let's assume that $M\not= D$.
Then $G/D$ is non-nilpotent.
 From the maximality of $M$ in $G$ we have that $M$ is $n$-modularly embedded in $G$ for some natural number $n\leq k$. Then $|G:M|=p$ and $|G/D|=pq^{n}$ for some $p, q\in \mathbb{P}$. By Burnside's theorem $G/D$ is solvable. By
Lemma~\ref{l1.1} $G/D=P/D\cdot M/D$, where $P/D$ is minimal normal subgroup in $G/D$, $(P/D)\cap (M/D)=1$, $P/D= C_ {G/D}(P/D)$. It follows that $|P/D|=|G:M|=p$ and $|M/D|=q^{n}$.  So $P/D\in\mathrm{Syl}_{p}(G/D)$ and $M/D\in\mathrm{Syl}_{q}(G/D)$. From $M/D\cong G/D/C_{G/D}(P/D)$
is isomorphic to a subgroup of $\mathrm{Aut}(Z_{p})\cong Z_{p-1}$ we have that $M/D$ is cyclic. Therefore  $(1) \Rightarrow (2)$.

If $(2)$ holds, then it is easy to verify that $M$ is $n$-modularly embedded in $G$. From the maximality of $M$ it follows that $M$ is k-submodular in $G$. Thus  $(2) \Rightarrow (1)$.
\end{proof}

\begin{lemma} \label{l2.3}
Let $G$ be a group.

$(1)$ If every maximal subgroup of a group $G$ is $k$-submodular in $G$, then $G$ is supersoluble.

$(2)$ If $M$ is a maximal subgroup of $G$ and $M$ is $k$-submodular in $G$, then $M$ is $\mathbb{P}$-subnormal and $\mathrm{K}$-$\mathbb{P}$-subnormal in $G$.
\end{lemma}

\begin{proof}
(1) Let $W$ be any maximal subgroup of $G$.
If $W=\mathrm{Core}_{G}(W)$, then $W\unlhd G$. Then $G=WP$ for some $P\in\mathrm{Syl}_{p}(G)$ and $p\in\pi(G:W)$. From the maximality of $W$ in $G$ it follows that $|G:W|=p$. Let's assume that $W\not=\mathrm{Core}_{G}(W)$. Then by Definition~1 $|G:W|$ is a prime number.
By Lemma \ref{l1.2} 
$G$ is supersoluble.

(2) The statement follows from proof (1) and Definition \ref{d1}.
\end{proof}

\begin{lemma} \label{l2.4}
Let $H$ be a subgroup of a group $G$. 

$(1)$ If $H$ is $k$-submodular in $G$, then $H^{x}$ is $k$-submodular in $G$ for all $x\in G$.

$(2)$ If $H\leq R\leq G$, $H$ is $k$-submodular in $R$ and $R$ is $k$-submodular in $G$, then $H$ is $k$-submodular in $G$.
\end{lemma}

\begin{proof}
Statements (1) and (2) follow from Definition \ref{d2}.
\end{proof}

\begin{lemma} \label{l2.5}
Let $H$ and $U$ be subgroups of a group $G$. 

$(1)$ If $H$ is $k$-submodular in $G$, then $H\cap U$ is $k$-submodular in $U$.

$(2)$ If $H$ is $k$-submodular in $G$ and $U$ is $k$-submodular in $G$, then $H\cap U$ is $k$-submodular in $G$.
\end{lemma}

\begin{proof}
(1) Suppose that $H$ is $k$-submodular in $G$. We can assume that $H\not= G$. There is a chain of subgroups (\ref{1_1}) such that $H_{i-1}$ is $n_{i}$-modularly embedded in $H_{i}$ for some $n_{i}\in\mathbb{N}$ such that $n_{i}\leq k$ and every $i = 1, \ldots , m$.
Consider the chain of subgroups
$$H\cap U = H_{0}\cap U \leq H_{1}\cap U \leq \cdots \leq H_{m-1}\cap U \leq H_{m}\cap U = U.$$

If $H_{i-1}\cap U \unlhd H_{i}\cap U$ for $i = 1, \ldots , m$, then $H\cap U$ is $k$-submodular in $U$.

Let's assume that $H_{j-1}\cap U\ntrianglelefteq H_{j}\cap U$ for some $j\in\{1, \ldots , m\}$. In that case $H_{j-1}\cap U\not= H_{j}\cap U$ and $H_{j-1}\ntrianglelefteq H_{j}$. Let $L= \mathrm{Core}_{H_{j}}(H_{j-1})$. Since $H_{j-1}$ is $n_{j}$-modularly embedded in $H_{j}$ for some natural number $n_{j}\leq k$, we have $|H_{j}:H_{j-1}| =p$ and $|H_{j}/L|=pq^{n_{j}}$ for some primes $p$ and $q$. Let $P\in \mathrm{Syl}_{p}(H_{j})$. Then $PL/L\in \mathrm{Syl}_{p}(H_{j}/L)$. From Lemma \ref{l2.2} it follows that $H_{j}/L=PL/L\cdot H_{j-1}/L$, $PL/L\unlhd H_{j}/L$, $H_{j-1}/L\in \mathrm{Syl}_{q}(H_{j}/L)$ and $H_{j-1}/L$ is a cyclic subgroup. We have two cases.

1. Let's assume that $p$ does not divide $|(H_{j}\cap U)L/L|$. By Sylow's theorem $(H_{j}\cap U)L/L\leq (H_{j-1}/L)^{xL}$ for some $x\in H_{j}$. Then $(H_{j}\cap U)L/L$ is cyclic. From
$$(H_{j-1}\cap U)/(L\cap U)\leq (H_{j}\cap U)/(L\cap U)\cong (H_{j}\cap U)L/L,$$
we conclude that $H_{j-1}\cap U\unlhd H_{j}\cap U$. We have obtained a contradiction with $H_{j-1}\cap U\ntrianglelefteq H_{j}\cap U$.

2. Suppose that $p$ divides $|(H_{j}\cap U)L/L|$. Since $PL/L\unlhd H_{j}/L$, we have $PL/L\leq (H_{j}\cap U)L/L$ by Sylow's theorem. Therefore $(H_{j}\cap U)L/L=PL/L\cdot (H_{j-1}\cap U)L/L.$
It follows that
$$p=|PL/L|=|(H_{j}\cap U)L/L:(H_{j-1}\cap U)L/L|=|(H_{j}\cap U):(H_{j-1}\cap U)|.$$
Let $D= \mathrm{Core}_{H_{j}\cap U}({H_{j-1}\cap U)}$. From $H_{j-1}\cap U\ntrianglelefteq H_{j}\cap U$ we have $D\not= H_{j-1}\cap U$ and $(H_{j}\cap U)/D=PD/D\cdot (H_{j-1}\cap U)/D$. Note that $(H_{j-1}\cap U)/(L\cap U)\cong (H_{j-1}\cap U)L/L\leq H_{j-1}/L$ is a cyclic $q$-group. Since $L\cap U\unlhd H_{j}\cap U$ and $L\cap U\leq H_{j-1}\cap U$, we have $L\cap U\leq D$. It follows that $(H_{j-1}\cap U)/D\cong (H_{j-1}\cap U)/(L\cap U)/D/(L\cap U)$ is a cyclic $q$-group. Thus  $|(H_{j-1}\cap U)/D|=q^{n}$ for some $n\in\mathbb{N}$ such that $n\leq n_{j}\leq k$. We have proved that $H_{j-1}\cap U$ is $n$-modularly embedded in $H_{j}\cap U$ for $n\leq k$. Thus  $H\cap U$ is $k$-submodular in $U$.

(2) The statement follows from (1) and Lemma \ref{l2.4}(2).
\end{proof}

\begin{lemma} \label{l2.6}
Let  $H$ be a subgroup of a group $G$ and $N\unlhd\ G$. 

$(1)$  If $H$ is $k$-submodular in $G$, then
$HN/N$ is $k$-submodular in $G/N$.

$(2)$ If $N\leq H$ and $H/N$ is $k$-submodular in $G/N$, then
$H$ is $k$-submodular in $G$.

$(3)$ The subgroup $HN$ is $k$-submodular in $G$ if and only if $HN/N$ is $k$-submodular in $G/N$.
\end{lemma}

\begin{proof}
(1) By Definition \ref{d2} there is a chain of subgroups (\ref{1_1}).
If $H_{i-1}N/N \unlhd H_{i}N/N$ for $i = 1, \ldots , m$, then $HN/N$ is $k$-submodular in $HN/N$.

Let's assume that $H_{i-1}N/N \ntrianglelefteq H_{i}N/N$ for some $i \in\{ 1, \ldots , m\}$. Then $H_{i-1}\ntrianglelefteq H_{i}$ and $H_{i-1}N\not= H_{i}N$. Let's denote $L=\mathrm{Core}_{H_{i}}(H_{i-1})$ and $T=\mathrm{Core}_{H_{i}N}(H_{i-1}N)$. Then $|H_{i}:H_{i-1}|=p$ and $|H_{i}/L|=pq^{n_{i}}$ for some $p, q\in \mathbb{P}$ and some $n_{i}\in\mathbb{N}$ such that $n_{i}\leq k$.
From
$$1\not=|H_{i}N/N:H_{i-1}N/N|=\frac{|H_{i}:H_{i-1}|}{|H_{i}\cap N:H_{i-1}\cap N|} $$
we have $H_{i}\cap N=H_{i-1}\cap N$ and $|H_{i}N/N:H_{i-1}N/N|=p$.
Since $H_{i}\cap N\unlhd H_{i}$ and $H_{i}\cap N\leq H_{i-1}$ it follows that $H_{i}\cap N\leq L$ and $|H_{i}N/LN|=|H_{i}/L|$. Then $pq^{n_{i}}=|H_{i}N/LN|=|H_{i-1}N/LN|\cdot p$ and $|H_{i-1}N/LN|=q^{n_{i}}$.
From $LN\leq T$ it follows that $|H_{i-1}N/T|=q^{n}$ for  $n\in\mathbb{N}$ and $n\leq n_{i}$.

Note that $T/N=\mathrm{Core}_{H_{i}N/N}(H_{i-1}N/N)\not=H_{i-1}N/N$. Therefore $|H_{i}N/N/T/N|=pq^{n}$.
We have proved that $H_{i-1}N/N$ is $n$-modularly embedded in $H_{i}N/N$ for $n\leq k$.
Thus  $HN/N$ is $k$-submodular in $G/N$.

(2) Let there be a chain of subgroups
$$H/N = H_{0}/N \leq H_{1}/N \leq \cdots \leq H_{m-1}/N \leq H_{m}/N = G/N$$
such that $H_{i-1}/N$ is $n_{i}$-modularly embedded in $H_{i}/N$ for some $n_{i}\in\mathbb{N}$ such that $n_{i}\leq k$ and every $i = 1, \ldots , m$. Since $N\unlhd H_{i}$ and $N\leq H_{i-1}$, Statement (2) follows from Lemma \ref{l2.1}.

(3) The statement follows from Statements (1) and (2).
\end{proof}

\begin{lemma} \label{l2.7}
Let $G$ be a soluble group. If a subgroup $H$ is $k$-submodular in $G$, then $H$ is $\mathfrak{U}_{k}$-subnormal in $G$.
\end{lemma}

\begin{proof}
By Definition \ref{d2} there is a chain  of subgroups (\ref{1_1}) such that $H_{i-1}$ is $n_{i}$-modularly embedded in $H_{i}$ for some natural number $n_{i}\leq k$ and every $i = 1, \ldots , m$. If $H= G$ then $H$ is $\mathfrak{U}_{k}$-subnormal in $G$. Assume that $H\not= G$. Then $H_{i-1}\not= H_{i}$ for some $i\in\{ 1, \ldots , m\}$.

Assume that  $H_{i-1}\unlhd H_{i}$. Since $G$ is soluble, we draw a composition series through $H_{i-1}$ and $H_{i}$:
$$H_{i-1} = R_{0}< R_{1}< \cdots < R_{s-1} < R_{s} = H_{i},$$
there $R_{j-1}$ is normal in $R_{j}$ and $|R_{j}:R_{j-1}|\in\mathbb{P}$, $j=1,\ldots, s$. Since $R_{j-1}=\mathrm{Core}_{R_{j}}(R_{j-1})$, we have $R_{j}/R_{j-1}\in\mathfrak{U}_{k}$. Then $R_{j}^{\mathfrak{U}_{k}}\leq R_{j-1}$.

Assume that $H_{i-1}\ntrianglelefteq H_{i}$.  Then  $\mathrm{Core}_{H_{i}}(H_{i-1})\not=H_{i-1}$. We have that $|H_{i} : H_{i-1}|=p$ and $|{H_{i}}/\mathrm{Core}_{H_{i}}(H_{i-1})|=pq^{n_{i}}$ for some $p, q\in\mathbb{P}$ and a natural number $n_{i}\leq k$. Thus  ${H_{i}}/\mathrm{Core}_{H_{i}}(H_{i-1})\in \mathfrak{U}_{k}$ and $H_{i}^{\mathfrak{U}_{k}}\leq H_{i-1}$.
Therefore $H$ is $\mathfrak{U}_{k}$-subnormal in $G$.
\end{proof}

\begin{lemma} \label{l2.8}
Let $H$ is a Sylow $p$-subgroup of a group $G$ and let $p$ be the greatest prime divisor of $|G|$. If $H$ is $k$-submodular in $G$ then $H$ is normal in $G$.
\end{lemma}

\begin{proof}
We proceed by induction on $|G|$. We may assume that $H\not = G$ and there exists a chain of subgroups (\ref{1_1}) such that $H_{i-1}$ is $n_{i}$-modularly embedded in $H_{i}$ for some natural number $n_{i}\leq k$ and $i = 1, \ldots , m$. By induction, $H$ is normal in $H_{m-1} = M$. By Definition~\ref{d2} either $M\unlhd G$ or $M\not= \mathrm{Core}_{G}(M)$, $|G:M| =r$ and $|G/\mathrm{Core}_{G}(M)|=rq^{n_{m}}$ for some $r, q\in\mathbb{P}$ and some natural number $n_{m}\leq k$. In the first case, $H$ is normal in $G$. Therefore, suppose that $M\not= \mathrm{Core}_{G}(M)$. In this case, $r\not= q$. If $N_{G}(H) \not= G$ then, by Sylow's theorem,
$|G : M| = |G : N_{G}(H)| = r \equiv 1 (\mathrm{mod}\ p)$. We get a contradiction to $r < p$. Therefore, $N_{G}(H) = G$.
\end{proof}

\begin{lemma} \label{l3.1}
Let $\mathfrak{F}$ be the class of all groups in which every Sylow subgroup is $k$-submodular. Let $G$ be a group. Then the following assertions hold.

$(1)$ $\mathfrak{N}\subseteq\mathfrak{F}\subseteq
\mathrm{w}\mathfrak{U}$.

$(2)$ If $G\in\mathfrak{F}$ and $N\unlhd G$, then $G/N\in\mathfrak{F}$.

$(3)$ If $G/N_{1}\in\mathfrak{F}$ and $G/N_{2}\in\mathfrak{F}$ for any $N_{i}\unlhd G$, $i = 1,2$, then $G/N_{1}\cap N_{2}\in\mathfrak{F}$.

$(4)$ A direct product of groups from $\mathfrak{F}$ lies in $\mathfrak{F}$.

$(5)$ If $G\in\mathfrak{F}$ and $U$ is a subgroup of $G$, then $U\in\mathfrak{F}$.
\end{lemma}

\begin{proof}
Statement (1) follows from Definitions \ref{d1} and \ref{d2}.

Statement (2) follows from Lemmas \ref{l2.4}, \ref{l2.6} and Sylow's theorem.

(3) Let $G$ be a group of minimal order such that $G/N_{i}\in\mathfrak{F}$, $N_{i}\unlhd G$, $i = 1,2$, but $G/N_{1}\cap N_{2}\not\in\mathfrak{F}$.
We can assume that $N_{1}\cap N_{2} = 1$. Let $Q\in\mathrm{Syl}_{q}(G)$. Then $QN_{i}/N_{i}\in\mathrm{Syl}_{q}(G/N_{i})$ and $QN_{i}/N_{i}$ is $k$-submodular in $G/N_{i}$ for $i=1, 2$. By Lemma~\ref{l2.6}(2) $QN_{i}$ is $k$-submodular in $G$ for $i=1, 2$. By Lemma~\ref{l2.5}(2) $QN_{1}\cap QN_{2}$ is $k$-submodular in $G$. By \cite[Theorem~A.6.4(b)]{DH} $QN_{1}\cap QN_{2}=Q(N_{1}\cap N_{2})=Q$ is $k$-submodular in $G$. Therefore $G=G/N_{1}\cap N_{2}\in\mathfrak{F}$. The contradiction thus obtained completes the proof of Statement (3).

Statement (4) follows from (3).

Statement (5) follows from Lemma \ref{l2.5}(1) and Sylow's theorem.
\end{proof}

\section{Proofs of main results}

\textbf{Proof of Theorem \ref{t3.1}.}

\begin{proof}
Assume that $(1)$ holds.
By Lemma \ref{l2.3} $G$ is supersoluble.

Let $H/K$ be any complemented chief factor of $G$. Then $G=HM$ for some subgroup $M$ of $G$ and $H\cap M=K$. Since $G/K$ is supersoluble, it follows that $|H/K|=p\in\mathbb{P}$. Then $M$ is maximal in $G$. Let $D=\mathrm{Core}_{G}(M)$.

If $M=D$, then $G= HM\leq C_{G}(H/K)$ and $G/C_{G}(H/K)=1$.

Let $M\not=D$. Then $G/D=HD/D\cdot M/D$, and $K=H\cap D$. From $HD/D\cong H/K$ it follows that $|HD/D|=p$ and $HD/D$ is a minimal normal subgroup of $G/D$.
By
Lemma~\ref{l1.1}(1) $HD/D= C_{G/D}(HD/D)$.
Since $G/D$ is non-nilpotent, by Lemma \ref{l2.2} $M/D\in \mathrm{Syl} _{q}(G/D)$ and $M/D$ is a cyclic group of order $q^{n}$ for some prime $q$ and some natural number $n\leq k$. Note that $HD\leq C_{G}(HD/D)$. Then $G/C_{G}(HD/D)\cong G/D/C_{G/D}(HD/D)\cong M/D$. From  \cite[Appendix B, Theorems 1 and 3]{Wei} it follows that $G/C_{G}(H/K)\cong \mathrm{Aut}_{G}(H/K)\cong \mathrm{Aut}_{G}(HD/D)\cong G/C_{G}(HD/D)$. Thus  $(1) \Rightarrow (2)$.


Assume that $(2)$ holds. Let $M_{1}$ and $M_{2}$ be maximal subgroups of $G$. If $G$ is primary cyclic group then $M_{1}=M_{2}$ is the only one maximal subgroup of $G$ and (3) holds. Let $G$ is not primary cyclic group and $M_{1}\not= M_{2}$.
Let $K=\mathrm{Core}_{G}(M_{1})$.

If $K\nleq M_{2}$ then $G=KM_{2}=M_{1}M_{2}$. From $G\in\mathfrak{U}$ it follows that $|G:M_{i}|=|M_{3-i}: M_{1}\cap M_{2}|\in\mathbb{P}$, $i=1, 2$. The subgroup $M_{1}\cap M_{2}$ is maximal in $M_{i}$, $i=1, 2$, so (3) it true.

Let $K\leq M_{2}$.
Then $K\not=M_{1}$. By
Lemma~\ref{l1.1}(1) $G/K=N/K\cdot M/K$, where $N/K$ is a minimal normal subgroup in $G/K$, $(N/K)\cap (M/K)=1$ and $N/K= C_{G/K }(N/K)$. Note that $|N/K|=p$ is a prime number because $G$ is supersoluble. By (2) we have $M/K\cong G/K/C_{G/K}(N/K)\cong G/C_{G}(N/K) $ is a cyclic group of order $q^{n}$ for some $q\in\mathbb{P}$ and a natural number $n\leq k$.

I. If $p$ divides $|M_{2}/K|$, then $N/K\leq M_{2}/K$ and $M_{2}/K=N/K\cdot(M_{1}\cap M_{2})/K$. From $|M_{2}|=p\cdot|M_{1}\cap M_{2}|$ we have that $M_{1}\cap M_{2}$ is maximal in $M_{2}$. From $G=M_{2}M_{1}$ and $|G: M_{2}|=|M_{1}:M_{1}\cap M_{2}|=q$ it follows that $M_{1}\cap M_{2}$ is maximal in $M_{1}$. So (3) it true.

II. If $p$ does not divide $|M_{2}/K|$, then $M_{2}/K$ is a $q$-group. By Sylow's theorem $M_{2}/K=M_{1}^{x}/K$ for some $x\in G$. Then $G=\langle M_{1}, M_{1}^{x}\rangle$ and $M_{1}\cap M_{1}^{x} < M_{1}$, $M_{1}\cap M_{1}^{x} < M_{1}^{x}$. Now $(M_{1}\cap M_{1}^{x})/K$ is normal subgroup in cyclic subgroup $M_{i}/K$, $i= 1, 2$. Then from $M_{1}\leq N_{G}(M_{1}\cap M_{1}^{x})$  and $M_{1}^{x}\leq N_{G}(M_{1}\cap M_{1}^{x})$ we have that $G=\langle M_{1}, M_{1}^{x}\rangle=N_{G}(M_{1}\cap M_{1}^{x})$. Therefore $K=M_{1}\cap M_{1}^{x}$. Since $M_{1}/K$ and $M_{1}^{x}/K$ are cyclic, $|M_{1}^{x}/K|=|M_{1}/K|=q^{n}$, we have that $K=M_{1}\cap M_{1}^{x}$ is an $n$-maximal subgroup in $M_{1}$ and $M_{1}^{x}$. We have proven $(2)\Rightarrow (3)$.

Assume that $(3)$ holds. Let $G$ be a group of the least order for which (3) holds and (1) does not. Then $G$ has a maximal subgroup $M$ such that is not $k$-submodular.
Let $N$ is a minimal normal subgroup of $G$. Since (3) holds for $G/N$ we have that (1) holds for $G/N$. By Lemma~\ref{l2.6}(2) and a choose of $M$ it follows that $\mathrm{Core}_{G}(M)=1$. By
Lemma~\ref{l1.1} $G=NM$, $N=C_{G}(N)$ is a unique minimal normal subgroup of $G$, $|N|=p^{s}$, $p\in \mathbb{P}$, $s\in \mathbb{N}$ and $N\cap M=1$. Let $R$ is a maximal subgroup of $M$. Then $|M:R|\in \mathbb{P}$ since $M\cong G/N\in\mathfrak{U}$ by Lemma~\ref{l2.3}(1). A subgroup $H=NR$ is maximal in $G$. By (3) $H\cap M$ is $m$-maximal subgroup of $H$ and of $M$, $|H:H\cap M |=q_{1}^{m}$ and $|M:H\cap M |=q_{2}^{m}$ for some $m\in\mathbb{N}$, $m\leq k$,  $q_{i}\in\mathbb{P}$, $i=1, 2$. Since $H\cap M=R$ is a maximal subgroup of $M$, we have that $m=1$. From $p^{s}=|H: R|=q_{1}^{m}$ it follows that $p=q_{1}$ and $s=1$. So $|N|=p$ and $G$ is supersoluble.

Since $M\cong G/C_{G}(N)$ is isomorphic to a subgroup of $\mathrm{Aut}(Z_{p})\cong Z_{p-1}$, the subgroup $M$ is cyclic. From $M=N_{G}(M)$ it follows that $M\not= M^{x}$ for some $x\in G$ and $G=\langle M, M^{x}\rangle$. Since $M\cap M^{x}\unlhd M$ and $M\cap M^{x}\unlhd M^{x}$ we have that $M\cap M^{x}\unlhd \langle M, M^{x}\rangle=G$ and so $M\cap M^{x}=1$. By (3) $1=M\cap M^{x}$ is an $n$-maximal subgroup of $M$ and $|M|=|M:M\cap M^{x}|=q^{n}$ for some $q\in\mathbb{P}$ and some $n\in\mathbb{N}$, $n\leq k$. By Lemma~\ref{l2.2} $M$ is $k$-submodular in $G$; a contradiction with choice $G$. We have proven $(3)\Rightarrow (1)$.
\end{proof}

\textbf{Proof of Theorem \ref{t3.2}.}

\begin{proof}
Assume that $(1)$ holds. If $M$ is a maximal subgroup of $A\leq G$ then $M$ is $k$-submodular in $G$  by (1). From Lemma~\ref{l2.5}(1) $M=M\cap A$ is $k$-submodular in $A$. Thus  $(1) \Rightarrow (2)$.

Assume that $(2)$ holds. By Lemma \ref{l2.3}, $G$ is supersoluble. Let $H/K$ be the chief factor of $G$. Then $|H/K|=p$ is prime. There is a Hall $p'$-subgroup $R/K$ in $G/K$ because $G/K$ is supersoluble. Then $G/K=P/K\cdot R/K$ for some $P/K\in \mathrm{Syl}_{p}(G/K)$. Let $A/K=H/K\cdot R/K$. Then $A=HR$, $H\cap R=K$ and $H/K$ are the chief factor in $A$. By Theorem \ref{t3.1}, $A/C_{A}(H/K)$ is either $1$ or a cyclic group of order $q^{n}$ for some $q\in\mathbb{P}$ and $n\in\mathbb{N}$ such than  $n\leq k$. Since $H\leq C_{A}(H/K)$, we have
$A=C_{A}(H/K)R$ and
$$A/C_{A}(H/K)\cong R/R\cap C_{A}(H/K)=R/C_{R}(H/K).$$
On the other hand, from $H/K\leq Z(P/K)$ it follows that $P\leq C_{G}(H/K)$. Then $G=PR=C_{G}(H/K)R$ and
$$G/C_{G}(H/K)\cong R/R\cap C_{G}(H/K)=R/C_{R}(H/K).$$
Thus  it has been proven that $(2) \Rightarrow (3)$.

Assume that $(3)$ holds. Let $K\leq G$ and $M$ is a maximal subgroup of $\langle M, K\rangle$. Let's denote $A=\langle M, K\rangle$. By (3) $A\in\mathfrak{U}$. Then $|A:M|=p$ for some $p\in\mathbb{P}$.

If $M\unlhd A$ we have that $A=MK$ and $|A:M|=|K:M\cap K|=p$. So (4) holds.

Let $M\not=\mathrm{Core}_{A}(M)=C$. By
Lemma~\ref{l1.1}(1) $A/C=N/C\cdot M/C$, where $N/C= C_{G/C }(N/C)$ is a unique minimal normal subgroup of $A/C$ and $(N/C)\cap (M/C)=1$. By (3)  $M/C\cong A/C/C_{A/C}(N/C)\cong A/C_{A}(N/C)$ is a cyclic group of order $q^{n}$ for some $q\in\mathbb{P}$ and a natural number $n\leq k$.

I. If $p$ divides $|KC/C|$ then $N/C\leq KC/C$. Since $A/C=N/C\cdot M/C$ then  $KC/C=N/C\cdot(M\cap K)C/C$. From which it follows that $|K:M\cap K|=p$ and (4) holds.

II. If $p$ does not divide $|KC/C|$ then $KC/C$ is a $q$-group. By Sylow's theorem $KC/C\leq M^{x}/C$ for some $x\in A$. Then $KC/C$ is cyclic. Since $M$ is maximal in $A=\langle M, K\rangle$ we have that $M\cap K\not= K$. From $(M\cap K)C/C\cong M\cap K/C\cap K\not= K/C\cap K\cong KC/C$ it follows that $|K: M\cap K|=|KC/C: (M\cap K)C/C|=q^{m}$ for some $m\in\mathbb{N}$ with $m\leq n\leq k$. We have proven $(3)\Rightarrow (4)$.

Assume that $(4)$ holds. Let $G$ be a group of the least order for which (4) holds and (1) does not. If $W$ is a maximal subgroup of $G$ then from the maximality $W$ in $\langle W, G\rangle$ and from (4) we have that $W=W\cap G$ is $n$-maximal subgroup of $G$ for some natural number $n\leq k$. Therefore $n=1$ and $|G:W|\in\mathbb{P}$.  By Lemma~\ref{l1.2} $G$ is supersoluble.

Since (4) holds for any $A< G$ then (1) holds for $A$. So $G$ has a maximal subgroup $M$ such that is not $k$-submodular in $G$. Then $M\not= \mathrm{Core}_{G}(M)=C$. By
Lemma~\ref{l1.1} $G/C$ has a unique minimal normal subgroup $N/C=C_{G/C}(N/C)$, $G/C=N/C\cdot M/C$ and $(N/C)\cap (M/C)=1$. From $G\in\mathfrak{U}$ it follows that $|N/C|=p\in\mathbb{P}$. Then $M/C$ is cyclic since $M/C\cong G/C/C_{G/C}(N/C)$.
Let's consider two cases

I. $|\pi(M/C|\geq 2$. Let $M_{1}/C$ be a maximal subgroup of $M/C$. Since $M/C\in\mathfrak{U}$ we have that $|M/C:M_{1}/C|=r\in\mathbb{P}$. Thence $M_{1}\not= C$. Let's denote $H/C=N/C\cdot M_{1}/C$. Then $H/\mathrm{Core}_{H}(M_{1})=N\mathrm{Core}_{H}(M_{1})/\mathrm{Core}_{H}(M_{1})\cdot M_{1}/\mathrm{Core}_{H}(M_{1})$ and $C\leq \mathrm{Core}_{H}(M_{1})$. From $(\mathrm{Core}_{H}(M_{1})/C)\cap (N/C)=C/C$ it follows that $\mathrm{Core}_{H}(M_{1})/C\leq C_{H/C}(N/C)\leq C_{G/C}(N/C)=N/C$. Therefore $\mathrm{Core}_{H}(M_{1})=C$. The statement (1) holds for $H$ and $M_{1}$ is maximal in $H$. Thence $M_{1}$ is $k$-submodular in $H$. Since $M_{1}/C\ntrianglelefteq H/C$, by Lemma~\ref{l2.2} $H/C$ is non-nilpotent, $|H/C|=pq^{n}$ for some prime $q$ and a natural number $n\leq k$, and $|M_{1}/C|=q^{n}$. Therefore $|M/C|=rq^{n}$. If $M_{2}/C$ is the maximal subgroup of $M/C$ and $|M/C: M_{2}/C|=q$, then reasoning as above, it is easy to show that $|M_{2}/C|=r$. Then $|M/C|=qr$ and $|G/C|=pqr$.

From $M/C=N_{G/C}(M/C)$ it follows that $M/C\not= M^{x}/C$ for some $x\in G$ and $G/C=\langle M/C, M^{x}/C\rangle$. Since $(M/C)\cap (M^{x}/C)\unlhd M/C$ and $(M/C)\cap (M^{x}/C)\unlhd M^{x}/C$ we have that $(M/C)\cap (M^{x}/C)\unlhd \langle M/C, M^{x}/C\rangle=G/C$ and so $(M/C)\cap (M^{x}/C)=C/C$. Then $M\cap M^{x}=C$ and $|M:M\cap M^{x}|=|M/C|=qr$. On the other hand (4) is satisfied for $G$. From the maximality $M^{x}$ in $\langle M^{x}, M\rangle$ we have that $M^{x}\cap M$ is $n$-maximal subgroup of $M$ and $|M: M^{x}\cap M|=s^{n}$ for some prime $s$ and a natural number $n\leq k$. We got a contradiction.

II. $|\pi(M/C)|< 2$. Since $G$ is not nilpotent, $|M/C|=q^{n}$ for some prime $q\not= p$ and some $n\in\mathbb{N}$. Let us show that $n\leq k$. Since $M\ntrianglelefteq G$ we have that $M\not = M^{y}$ and $G=\langle M, M^{y} \rangle$ for some $y\in G$.
By (3) $M\cap M^{y}$ is $m$-maximal subgroup of $M$ and $|M: M\cap M^{y}|=q^{m}$ for some natural number $m\leq k$. Since $M/C$ is cyclic it follows that $(M\cap M^{y})/C=(M/C)\cap (M^{y}/C)\unlhd M/C$ and $(M\cap M^{y})/C=(M/C)\cap (M^{y}/C)\unlhd M^{y}/C$. Thence $G/C=\langle M/C, M^{y}/C \rangle\leq N_{G/C}((M\cap M^{y})/C)$. So $M\cap M^{y}\unlhd G$. From which it follows that $M\cap M^{y}=C$. Then $q^{n}=q^{m}$ and $n= m\leq k$. Therefore $M$ is $k$-submodular in $G$; a contradiction with choice $G$. We have proven $(4)\Rightarrow (1)$.
\end{proof}

\textbf{Proof of Theorem \ref{t3.3}.}

\begin{proof}
(1) By Lemma \ref{l2.3}(1) $\mathfrak{X}$ consists of supersoluble groups.

Let $G\in \mathfrak{X}$ and $N\unlhd G$. For any maximal subgroup $M/N$ of $G/N$, the subgroup $M$ is maximal in $G$. Then $M$ is $k$-submodular in $G$ by the choice of $G$. By Lemma~\ref{l2.6}(1) $M/N$ is $k$-submodular in $G/N$. Thus $G/N\in \mathfrak{X}$ and $\mathfrak{X}$ is a homomorph.

Let $G/\mathrm{Core}_{G}(M)\in \mathfrak{X}$ for any maximal subgroup $M$ of $G$. Then $M/\mathrm{Core}_{G}(M)$ is $k$-submodular in  $G/\mathrm{Core}_{G}(M)$. By Lemma~\ref{l2.6}(2) $M$ is $k$-submodular in $G$. We conclude that $G\in \mathfrak{X}$. Thus $\mathfrak{X}$ is a Schunck class.


(2) It is obvious that $\mathfrak{Y}$ is hereditary.

Let us prove that $\mathfrak{Y}$ is a homomorph. Take $G\in \mathfrak{Y}$ and $N\unlhd G$. Let $A/N$ be any subgroup of $G/N$, and let $W/N$ be a maximal subgroup of $A/N$. From the maximality of $W$ in $A\leq G\in \mathfrak{Y}$ it follows that $W$ is $k$-modular in $A$. By Lemma~\ref{l2.6}(1) $W/N$ is $k$-submodular in $A/N$. Hence $G/N\in \mathfrak{Y}$ and $\mathfrak{Y}$ is a homomorph.

Suppose that there are groups $G$ such that $G/N_{i}\in\mathfrak{Y}$, $i=1, 2$, but $G/N_{1}\cap N_{2}\not\in\mathfrak{Y}$. Let us consider a group $G$ of minimal order with this property.

We can assume that $N_{1}\cap N_{2}=1$ and  $N_{i}$ is a minimal normal subgroups of $G$, $i=1, 2$.
Since $\mathfrak{Y}\subseteq \mathfrak{U}$ and $\mathfrak{U}$ are a formation, we have $G\in\mathfrak{U}$. Then $N_{i}$ is a cyclic group of prime order, $i=1, 2$.

Let $H/K$ be a chief factor of $G$. From $K\leq K(H\cap N_{1})\leq H$ it follows that either $K=K(H\cap N_{1})$ or $H=K(H\cap N_{1})$.

1. If $K=K(H\cap N_{1})$ we have that $H/K = H/K(H\cap N_{1}) = H/H\cap KN_{1} \cong HN_{1}/N_{1}/KN_{1}/N_{1}$ is a chief factor of $G/N_{1}$. Then  $G/C_{G}(H/K)\cong G/N_{1}/C_{G/N_{1}}(HN_{1}/N_{1}/KN_{1}/N_{1})$ by \cite[Appendix B, Theorems 1-3]{Wei}.

2. If $H=K(H\cap N_{1})$ then $H\cap N_{1}\not=K\cap N_{1}$. Since $1\leq K\cap N_{1} < H\cap N_{1}\leq N_{1}$ it follows that $1= K\cap N_{1}$ and $H\cap N_{1}= N_{1}$. So $H=KN_{1}$. We have that $H/K = KN_{1}/K \cong N_{1}\cong N_{1}N_{2}/N_{2}$ is a chief factor of $G/N_{2}$. Then $G/C_{G}(H/K)\cong G/N_{2}/C_{G/N_{2}}(N_{1}N_{2}/N_{2})$ by \cite[Appendix B, Theorems 1-3]{Wei}.

Since $G/N_{i}\in\mathfrak{Y}$, it follows that Statement (2) of Theorem~\ref{t3.2} holds for $G/N_{i}$, $i=1, 2$. Then Statement (2) of Theorem~\ref{t3.2} is valid for $G$ in view of 1 and 2 proved above. By (1) of Theorem~\ref{t3.2} $G\in\mathfrak{Y}$; a contradiction. Thus  $\mathfrak{Y}$ is a hereditary formation and Statement (2) of Theorem is proven.

(3) Suppose that $G\in\mathfrak{X}$. Since $\mathfrak{X}$ is a homomorph, we have that $G/\Phi(G)\in\mathfrak{X}$. Then Statement (2) of Theorem~\ref{t3.1} is satisfied for $G/\Phi(G)$. Since any chief factor of $G/\Phi(G)$ is complemented, Statement (2) of Theorem~\ref{t3.2} is true for $G/\Phi(G)$. By Statement (1) of Theorem~\ref{t3.2} $G/\Phi(G)\in\mathfrak{Y}$.

If $G/\Phi(G)\in\mathfrak{Y}$, then $G/\mathrm{Core}_{G}(M)\in\mathfrak{Y}\subseteq\mathfrak{X}$ for any maximal subgroup $M$ of $G$. By (1) $G\in\mathfrak{X}$.
\end{proof}

\textbf{Proof of Theorem \ref{t3.5}.}

\begin{proof}
By Lemma \ref{l3.1} $\mathfrak{F}=(G : $  every Sylow subgroup of $G$ is $k$-submodular in $G)$ is a hereditary formation.
Since $\mathfrak{K}=\mathfrak{U}\cap\mathfrak{F}$ we have that $\mathfrak{K}$ is a hereditary formation.

Let us prove that $\mathfrak{K}$ is saturated.
Let $G$ be a group of minimal order for which $G/\Phi(G)\in\mathfrak{K}$ and $G\not\in\mathfrak{K}$. Then there exists a Sylow $q$-subgroup $Q$ of $G$ that is not $k$-submodular in $G$. From $G/\Phi(G)\in\mathfrak{U}$ it follows that $G\in\mathfrak{U}$. Then a Sylow $p$-subgroup is normal in $G$ for the largest prime $p\in\pi(G)$.
Let $N$ be a minimal normal subgroup of $G$.  By \cite[Theorem~A.9.2(e)]{DH}  $\Phi(G)N/N\leq \Phi(G/N)$. Since $G/\Phi(G)N\in\mathfrak{K}$, we have that $(G/N)/\Phi(G/N)\in\mathfrak{K}$. From $|G/N| < |G|$, it follows
$G/N\in\mathfrak{K}$. Then $N$ is a unique minimal normal subgroup of $G$, $N\leq \Phi(G)\leq F(G)$ and $F(G)\in\mathrm{Syl}_{p}(G)$. We note that $|N|=p$. Since $QN/N\in\mathrm{Syl}_{q}(G/N)$ and $QN/N\nleqslant\Phi(G/N)$, we have that $QN/N\not=N/N$ and $QN/N$ is $k$-submodular  in $G/N$. By Lemma~\ref{l2.6}(2) $QN$ is $k$-submodular in $G$. Then $Q\ntrianglelefteq QN$ and $q\not= p$ by choice of $Q$.

Let $B=QN$. Then $T=\mathrm{Core}_{B}(Q)\not=Q$. We have that $B/T=NT/T\cdot Q/T$ and $C_{B/T}(NT/T)=NT/T$. Therefore $Q/T\cong B/T/C_{B/T}(NT/T)$ is isomorphic to a subgroup from $Z_{p-1}$. Then $|Q/T|=q^{n}$ for some $n\in\mathbb{N}$.

1. Assume that $G/N = QF(G)/N$. By Lemma~\ref{l2.7} $G/N\in\mathrm{w}\mathfrak{U}_{k}$.
By Lemma~\ref{l1.6}(3) 
$G/N/F(G/N)\in\mathcal{A}_{k}$. Since $F(G/N)=F(G)/N$ we have that $Q\cong G/F(G)\in\mathcal{A}_{k}$. So $n\leq k$. By Lemma~\ref{l2.2} $Q$ is $k$-submodular in $QN$. By Lemma~\ref{l2.4}(2) $Q$ is $k$-submodular in $G$; a contradiction.

2. Assume that $G/N \not= QF(G)/N=H/N$.  From the hereditary $\mathfrak{K}$
it follows $H/N\in\mathfrak{K}$. By Lemma~\ref{l2.7} $H/N\in\mathrm{w}\mathfrak{U}_{k}$.
By Lemma~\ref{l1.6}(2) 
$\mathrm{w}\mathfrak{U}_{k}$ is a hereditary saturated formation. Since $H$ is a Hall subgroup in $G$, by \cite[Corollary 16.2.3]{Shem} from $H/H\cap\Phi(G)\in\mathrm{w}\mathfrak{U}_{k}$ it follows that $H\in\mathrm{w}\mathfrak{U}_{k}$. Then $B=QN\in\mathrm{w}\mathfrak{U}_{k}$.
By Lemma~\ref{l1.6}(3) 
$B/F(B)\in\mathcal{A}_{k}$. Then $Q/(F(B)\cap Q)\cong QF(B)/F(B)\in \mathcal{A}_{k}$. We note that $F(B)=N(F(B)\cap Q)$ and $(F(B)\cap Q)\leq T=\mathrm{Core}_{B}(Q)$. So $Q/T\in\mathcal{A}_{k}$ and $n\leq k$. By Lemmas~\ref{l2.2} and \ref{l2.4}(2) $Q$ is $k$-submodular in $G$; a contradiction.
This $\mathfrak{K}$ is saturated.

Let us prove that $LF(h)= \mathfrak{K}$.
By \cite[Proposition~IV.3.14 and Theorem~IV.4.6]{DH} $ LF(h)$ is a hereditary saturated formation.

Show that $LF(h)\subseteq \mathfrak{K}$. Let $G$ be a group of minimal order in $ LF(h)\setminus\mathfrak{K}$. Let $N$ be a minimal normal subgroup of $G$. By choice of $G$ we have that $G/N\in\mathfrak{K}$, $\Phi(G)=1$ and $N$ is a unique minimal normal subgroup of $G$, since $ LF(h)$ and $\mathfrak{K}$ are saturated formation.
From $h(p)\subseteq\mathfrak{A}(p-1)$ it follows that $G\in\mathfrak{U}$, so $|N|=p$ for some prime $p$.
 From $G\not\in\mathfrak{K}$ it follows that $G$ has a Sylow $q$-subgroup $Q$ that is not $k$-submodular in $G$. Then $QN/N\in\mathrm{Syl}_{q}(G/N)$. From $G/N\in\mathfrak{K}$  by Lemma~\ref{l2.6}(2) we have that $QN$ is $k$-submodular in $G$. So $Q$ is not normal in $QN$ and $q\not=p$.

If $G\not=QN$, then $QN\in LF(h)$ since $LF(h)$ is hereditary. By choose of $G$ $QN\in \mathfrak{K}$. Thus $Q$ is $k$-submodular in $QN$ and $Q$ is $k$-submodular in $G$; a contradiction.

Let $G=QN$. Since $\mathrm{Core}{G}(Q)=1$ and $G\in LF(h)$ we have that $Q\cong G/N=G/C_{G}(N)\in h(p)$. Therefore $Q\in \mathfrak{A}(p-1)_{k}$ and $Q$ is cyclic. Then $Q$ is $k$-submodular in $G$; a contradiction. Thus $ LF(h)\subseteq \mathfrak{K}$.

Prove that $\mathfrak{K}\subseteq LF(h)$.
Let $G$ be a group of minimal order in $\mathfrak{K}\setminus LF(h)$. We denote by $N$ a minimal normal subgroup of $G$. Then $G/N\in\mathfrak{K}$. Since $\mathfrak{K}$ and $ LF(h)$ are saturated formations, we have that $\Phi(G) = 1$ and $N$ is a unique minimal normal subgroup of $G$. Then $G = NM$, where $M$ is a maximal subgroup of $G$, $\mathrm{Core}_{G}(M)=1$, $M\cap N=1$ and $N = C_{G}(N)$. From $G\in\mathfrak{K}\subseteq\mathfrak{U}$ it follows that $|N|=p$ for some prime $p$. Then $M\cong G/C_{G}(N)$ is isomorphic to a subgroup from $Z_{p-1}$. Let $R\in\mathrm{Syl}_{q}(M)$.

If $G\not= RN=D$ then $D\in \mathfrak{K}$, since $\mathfrak{K}$ is hereditary. By choose of $G$, $D\in LF(h)$. So $R\cong D/C_{D}(N)\in h(p)$, since $N=C_{D}(N)$. Thus $R\in\mathfrak{A}(p-1)_{k}$. So $G/C_{G}(N)\cong M\in h(p)$ and $G\in LF(h)$; a contradiction.

Let $G= RN$. Then $M^{x}=R$ for some $x\in G$.  Since $G\in\mathfrak{K}$, $R$ is $k$-submodular in $G$. From maximality of $R$ in $G$ it follows that $R$ is $n$-modularly embedded in $G$ for some natural $n\leq k$. So $G/C_{G}(N)\in h(p)$ and $G\in LF(h)$; a contradiction. Thus $\mathfrak{K}\subseteq LF(h)$. This means that the equality $\mathfrak{K}= LF(h)$ is proven.
\end{proof}

\textbf{Proof of Theorem \ref{t3.6}.}

\begin{proof}
By Lemma \ref{l3.1} $\mathfrak{F}$ is a hereditary formation.

Prove that $\mathfrak{F}$ is saturated. Let $G$ be a group of least order such that $G/\Phi(G)\in\mathfrak{F}$ and $G\not\in\mathfrak{F}$. Since $\mathfrak{F}\subseteq\mathrm{w}\mathfrak{U}$ and by Lemma \ref{l1.5}(2) $\mathrm{w}\mathfrak{U}$ is saturated formation, we have that $G\in\mathrm{w}\mathfrak{U}$. So $G$ is soluble. Let $N$ be a minimal normal subgroup of $G$. By choose of $G$, it follows that $G/N\in\mathfrak{F}$. Then $N$ is a unique minimal normal subgroup of $G$. Hence, $N\leq \Phi(G)< F(G)$ and $F(G)$ is $p$-group for some prime $p$. Since $G$ is Ore dispersive, we have that $p$  is the greatest prime divisor of $|G|$ and $F(G)\in\mathrm{Syl}_{p}(G)$. From $G\not\in\mathfrak{F}$ it follows that $G$ has a Sylow $q$-subgroup $Q$ which is not $k$-submodular in $G$. Since $G/N\in\mathfrak{F}$ we have that $QN/N$ is $k$-submodular in $G/N$. By Lemma 2.6(3) $QN$ is $k$-submodular in $G$. So $q\not=p$ and $\mathrm{Core}_{QN}(Q)\not= Q$.

1. $|\pi(G)| = 2$. Then $|\pi(G/N)| = 2$ and
$G/N = QN/N \cdot F(G)/N$. By Lemma~\ref{l3.1}(1) $G/N\in\mathrm{w}\mathfrak{U}$.  So $QN/N$ and $F(G)/N$ are $\mathbb{P}$-subnormal in $G/N$.
 By Lemma~\ref{l1.4} 
$G/N\in\mathfrak{U}$. So $G/N\in\mathfrak{K}$. By Theorem~\ref{t3.5} $\mathfrak{K}$ is saturated. Then $G\in\mathfrak{K}\subseteq\mathfrak{F}$; a contradiction.

2. $|\pi(G)| > 2$. Then $QN/N$ is contained in a Hall $\{p, q\}$-subgroup $H/N$ of $G/N$ and $H/N = QN/N \cdot F(G)/N$.
By the hereditary of $\mathfrak{F}$ $H/N\in\mathfrak{F}$. Therefore  $QN/N$ and $F(G)/N$ are $k$-submodular in $H/N$.
By Lemma~\ref{l1.4} 
$H/N\in\mathfrak{U}$, so  $H/N\in\mathfrak{K}$. Since by Theorem~\ref{t3.5} $\mathfrak{K}$ a local formation and $H = QF(G)$ is a Hall $\{p, q\}$-subgroup of $G$, we can apply Corollary 16.2.3 of \cite{Shem}. From $H\Phi(G)/\Phi(G)\cong H/H \cap \Phi(G)\cong H/N/H\cap\Phi(G)/N \in \mathfrak{K}$ it follows that $H\in \mathfrak{K}$. Then $QN\in \mathfrak{K}$ and $Q$ is $k$-submodular in $QN$. So $Q$ is
$k$-submodular in $G$. This contradiction completes the proof that $\mathfrak{F}$ is saturated.

Let us prove that $LF(f)= \mathfrak{F}$.

It is easy to verify that $f(p)$ is a hereditary formation. By \cite[Proposition~IV.3.14 and Theorem~IV.4.6]{DH} $LF(f)$ is hereditary saturated formation.
By Lemma~\ref{l1.5}
$LF(f)\subseteq\mathrm{w}\mathfrak{U}$ and $LF(f)$ consists of Ore dispersive groups.

Show that $LF(f)\subseteq \mathfrak{F}$. Let $G$ be a group of least order in $LF(f)\setminus \mathfrak{F}$. Since $LF(f)$ and $\mathfrak{F}$ are saturated formations, we conclude that $G$ has a unique minimal normal subgroup $N$ and $\Phi(G) = 1$. From $G\in\mathfrak{S}$ it follows that $N = C_{G}(N)$ is an elementary abelian $p$-group for some prime $p$. The Ore dispersiveness of $G$ implies that $N\in \mathrm{Syl}_{p}(G)$ and $p$ is the greatest prime divisor of $|G|$. Since $G\not\in\mathfrak{N}$, $|\pi(G)|\geq 2$.

If $QN\not=G$ for all $Q\in\mathrm{Syl}_{q}(G)$, $q\not= p$, then from $QN\in LF(f)$ it follows that  $QN\in\mathfrak{F}$. So $Q$ is $k$-submodular in $QN$. Since $QN/N\in\mathrm{Syl}_{q}(G/N)$ and $G/N\in\mathfrak{F}$, we conclude that $QN/N$ is $k$-submodular in $G/N$. By Lemma~\ref{l2.6}(2) $QN$ is $k$-submodular in $G$. So $Q$ is $k$-submodular in $G$. Hence $G\in\mathfrak{F}$; a contradiction.

Let $QN=G$ for some $Q\in\mathrm{Syl}_{q}(G)$. Then $Q$ is a maximal subgroup of $G$ and $\mathrm{Core}_{G}(Q)=1$. From $G\in LF(f)$ it follows that $Q\cong G/C_{G}(N)\in f(p)$. So $Q\in\mathfrak{A}(p-1)_{k}$ and $Q$ is cyclic.

Since $G\in LF(f)\subseteq\mathrm{w}\mathfrak{U}$ it follows that $Q$ is $\mathbb{P}$-subnormal in $G$, we have that $|G:Q|=p$. By Lemma~\ref{l2.2} $Q$ is $k$-submodular in $G$. Hence $G\in\mathfrak{F}$ is a contradiction to the choice of $G$. Thus  $LF(f)\subseteq \mathfrak{F}$.

Prove that $\mathfrak{F}\subseteq LF(f)$. Let $G$ be a group of least order in $\mathfrak{F}\setminus LF(f)$.
Since $\mathfrak{F}$ and $LF(f)$ are saturated formations, we have that $G/N\in LF(f)$, $G$ has a unique minimal normal subgroup $N$, $\Phi(G) = 1$. Then $G=NM$, where $M$ is a maximal subgroup of $G$, $N\cap M = 1$, $N = C_{G}(N)$ is an $p$-group for some prime $p$, because $G\in\mathfrak{S}$.
From $G\in\mathrm{w}\mathfrak{U}$ we conclude that $G$ is Ore dispersive. Then $p$ is the greatest prime divisor of $|G|$ and $N\in \mathrm{Syl}_{p}(G)$.
From $G\not\in LF(f)$ it follows that $G/C_{G}(H/K)\not\in f(p)$ for some chief factor $H/K$ of $G$. Then $H=N$ and $K=1$ because $G/N\in LF(f)$. So $G/C_{G}(N)\not\in f(p)$.

Let $R\in\mathrm{Syl}_{q}(M)$. Then $R\in\mathrm{Syl}_{q}(G)$.

Suppose that $D=RN\not= G$. Since $N\leq F(D)$, $F(D)=(R\cap F(D))N$. So $R\cap F(D)\in\mathrm{Syl}_{q}(F(D)$.
Then $R\cap F(D)\leq C_{D}(N)\leq C_{G}(N)=N$ and $F(D)=N$. From $D\in\mathfrak{F}\subseteq\mathrm{w}\mathfrak{U}$ and
Lemma~\ref{l1.4} 
we have that $D\in\mathfrak{K}$.
By Theorem~\ref{t3.5} $D/C_{D}(L/T)\in h(p)$ for every chief factor $L/T$ of $D$. By \cite[Corollary 4.1.1]{Shem} $F(D)=\cap C_{D}(L/T)$ for every chief factor $L/T$ of $D$.  So $R\cong D/N\in h(p)$. Then $G/C_{G}(N)=G/N\cong M\in f(p)$; a contradiction.

Suppose that $RN = G$.
By Lemma~\ref{l1.4} 
$G\in\mathfrak{U}$. Then $|N|=p$ and $R\cong G/N=G/C_{G}(N)\cong L\leq Z_{p-1}$. From $G\in\mathfrak{F}$ and $R=M$ it follows that $R$ is $n$-modularly embedded in $G$, $n\leq k$. Therefore $G/C_{G}(N)\cong R\in \mathfrak{A}(p-1)_{k}$ and $G/C_{G}(N)\in f(p)$; a contradiction.

 Hence $\mathfrak{F}\subseteq LF(f)$. It has been proven that $\mathfrak{F}= LF(f)$.
\end{proof}

\textbf{Proof of Proposition \ref{p3.1}.}

\begin{proof}
(1) Since $\mathfrak{K}\subseteq\mathfrak{U} $, by Lemma~\ref{l2.7} $\mathfrak{K}\subseteq\mathfrak{U}\cap\mathrm{w}\mathfrak{U}_{k}$.

Show that $\mathfrak{U}\cap\mathrm{w}\mathfrak{U}_{k}\subseteq \mathfrak{K}$.
Let $G$ be a group of minimal order in
$\mathfrak{U}\cap\mathrm{w}\mathfrak{U}_{k}\setminus \mathfrak{K}\not=\varnothing$.
Let $N$ be a minimal normal subgroup of $G$. Then $|N|=p$ for some prime $p$.
From $G\not\in\mathfrak{K}$ it follows that there exists a Sylow $q$-subgroup
$Q$ of $G$ that $Q$ is not $k$-submodular in $G$. Since $QN/N\in\mathrm{Syl}_{q}(G/N)$ and $G/N\in\mathfrak{K}$ we have that $QN$ is $k$-submodular in $G$ by Lemma~\ref{l2.6}(3). If $QN\not= G$ then from $QN\in \mathfrak{U}\cap\mathrm{w}\mathfrak{U}_{k}$ it follows that $QN\in\mathfrak{K}$. By Lemma~\ref{l2.4}(2), $Q$ is $k$-submodular in $G$; a contradiction.

Let $G=QN$. Then $q\not=p$. Since $\mathfrak{U}\cap\mathrm{w}\mathfrak{U}_{k}$ and $\mathfrak{K}$ are hereditary saturated formations we have that $N$ is a unique minimal normal subgroup of $G$ and $\Phi(G)=1$.
Then  $\mathrm{Core}_{G}(Q)=1$. By Lemma~\ref{l1.1} $N=C_{G}(N)=F(G)$. From $G/C_{G}(N)\cong W\leq Z_{p-1}$ it follows that $Q$ is cyclic and $|Q|=q^{n}$ for some $n\in\mathbb{N}$.
By Lemma~\ref{l1.6}(3)
$Q\cong G/F(G)\in\mathcal{A}_{k}$. Therefore $n\leq k$ and $Q$ is $k$-submodular in $G$; a contradiction. Thus $\mathfrak{U}\cap\mathrm{w}\mathfrak{U}_{k}\subseteq\mathfrak{K}$. This means that the equality $\mathfrak{K}=\mathfrak{U}\cap\mathrm{w}\mathfrak{U}_{k}$ is proven.

(2) Note that $\mathfrak{F}$ and $\mathrm{w}\mathfrak{K}$ consist of soluble groups. Suppose that $\mathfrak{F}\setminus\mathrm{w}\mathfrak{K}\not=\varnothing$, and choose a group $G$ of minimal order in $\mathfrak{F}\setminus\mathrm{w}\mathfrak{K}$. Since $\mathfrak{F}$ is a formation, we have that $G/N\in\mathrm{w}\mathfrak{K}$ for any minimal normal subgroup $N$ of $G$. By Lemma~\ref{l1.10}, $\mathrm{w}\mathfrak{K}$ is a saturated formation. Therefore $N$ is a unique minimal normal subgroup of $G$ and $\Phi(G)=1$. Let $H\in\mathrm{Syl}_{q}(G)$ such that $H$ is not $\mathfrak{K}$-subnormal in $G$. Since $HN/N\in\mathrm{Syl}_{q}(G/N)$ we have that $HN/N$ $\mathfrak{K}$-$sn$ $G/N$. By Lemma~\ref{l1.8}(2) $HN$ $\mathfrak{K}$-$sn$ $G$. Then $\mathrm{Core}_{G}(H)=1$ and $q\not= p$ for $p\in\pi(N)$. If $HN\not=G$ then $HN\in\mathrm{w}\mathfrak{K}$. By Lemma~\ref{l1.8}(1) $H$ $\mathfrak{K}$-$sn$ $G$; a contradiction. Let $HN=G$. Then in $G$, $H$ is maximal and $k$-submodular. By Definitions~\ref{d1}, \ref{d2}, $|N|=|G:H|=p$ and $|G|=|G/\mathrm{Core}_{G}(H)|=pq^{n}$ for some natural number $n\leq k$. Thus $G\in\mathfrak{K}$ and $G^{\mathfrak{K}}=1$. By Lemma~\ref{l1.9}(1) $H$ $\mathfrak{K}$-$sn$ $G$; a contradiction. Thus $\mathfrak{F}\subseteq\mathrm{w}\mathfrak{K}$.

Now suppose that $\mathrm{w}\mathfrak{K}\setminus\mathfrak{F}\not=\varnothing$, and choose a group $G$ of minimal order in $\mathrm{w}\mathfrak{K}\setminus\mathfrak{F}$. By Lemma~\ref{l1.10} $\mathrm{w}\mathfrak{K}$ is hereditary. By the choose $G$, $\mathfrak{F}$ contains every proper subgroup of $G$, $G/N\in\mathfrak{F}$ where $N$ is a unique minimal normal subgroup of $G$, and $\Phi(G)=1$. Let $H\in\mathrm{Syl}_{q}(G)$ such that $H$ is not $k$-submodular in $G$. From choose of $G$ and Lemmas~\ref{l2.6}(3), \ref{l2.4}(2) it follows that $HN=G$. From the maximality $H$ in $G\in\mathrm{w}\mathfrak{K}$ we have that $G^{\mathfrak{K}}\leq\mathrm{Core}_{G}(H)=1$. Then $G\in\mathfrak{K}\subseteq\mathfrak{F}$; a contradiction. Hence $\mathrm{w}\mathfrak{K}\subseteq\mathfrak{F}$. It follows that $\mathrm{w}\mathfrak{K}=\mathfrak{F}$.
\end{proof}

\textbf{Proof of Theorem  \ref{t3.6_1}.}

\begin{proof} Let $G$ be a counterexample of least order to the theorem. Then, by the Wielandt–Kegel theorem \cite[Theorem~VI.4.3]{Hup}, $G$ is soluble. Let $N$ be a minimal normal subgroup of $G$. Then $AN/N\cong A/A\cap N\in\mathfrak{N}$ and $BN/N\cong B/B\cap N\in\mathfrak{N}$. By Lemma~\ref{l2.6}(1), $AN/N$ and $BN/N$ are $k$-submodular in $G/N$. The choice of $G$ implies that $G/N$ is supersoluble and every Sylow subgroup is $k$-submodular in $G/N$. By Theorem~\ref{t3.5}, we conclude that $N$ is the only minimal normal
subgroup of $G$ and $\Phi(G) = 1$. Then $G = MN$, where $M$ is a maximal subgroup of $G$, $M\cap N = 1$, $N = C_{G}(N)$ is a $p$-group for some prime $p$. By Lemma~\ref{l1.7} $A\cap B = 1$. From $N\leq A \cup B$ it
follows that either $N\leq A$ or $N\leq B$. Without loss of generality, we may assume that $N\leq A$. Then, by Lemma~\ref{l1.3}(3), $A$ is a $p$-subgroup and $B$ is a $p'$-subgroup.

Let $q$ be the greatest prime divisor of $|G|$.

If $q\not= p$ then $Q\leq B$ where $Q$ is some Sylow $q$-subgroup of $G$. From
$Q\unlhd B$ and the $k$-submodularity of $B$ in $G$ it follows that $Q$ is $k$-submodular in $G$. By Lemma~\ref{l2.8}  $Q\unlhd G$.
Then $N\leq Q$. We get a contradiction to $q\not= p$.

Thus  $q = p$. By Lemma~\ref{l2.8} $A\unlhd G$. By
Lemma~\ref{l1.1}(2) $O_{p}(M) = 1$. Then $M\cap A \leq O_{p}(M) = 1$ and $A = N\in\mathrm{Syl}_{p}(G)$, while $B$ is a maximal subgroup of $G$ and $\mathrm{Core}_{G}(B) = 1$. From Definitions~\ref{d1} and \ref{d2} it follows that $|G:B| =p$ and $|G|=|G/\mathrm{Core}_{G}(B)|=pq^{n}$ for some primes $p$ and $q$, $n\in\mathbb{N}$ and $n\leq k$. Therefore  $G$ is supersoluble and every Sylow subgroup is $k$-submodular in $G$; a contradiction with the choice of $G$. The theorem is proved.
\end{proof}

\medskip

This work was supported by the Ministry of Education of the Republic of Belarus (Grant no. 20211750"Convergence-2025").



T.\ I. Vasilyeva,

Department of Higher Mathematics,
Belarusian State University of Transport,
Gomel, Belarus.

E-mail: tivasilyeva@mail.ru


\end{document}